\documentclass[a4paper]{article}
\usepackage{latexsym}
\usepackage{amsmath,amssymb,amsfonts}
\renewcommand{\baselinestretch}{1}
\usepackage[top=2cm,bottom=2cm,left=2.3cm,right=2.3cm]{geometry}

\newtheorem{prethm}{{\bf Theorem}}

\newenvironment{thm}{\begin{prethm}{\hspace{-0.5
               em}{\bf.}}}{\end{prethm}}

\newtheorem{prepro}[prethm]{Proposition}

\newtheorem{prelem}[prethm]{Lemma}

\newtheorem{precor}[prethm]{Corollary}

\newtheorem{preque}[prethm]{Question}

\newtheorem{precon}[prethm]{Conjecture}

\newtheorem{preremark}{{\bf Remark}}

\newenvironment{rem}{\begin{preremark}\em{\hspace{-0.5
              em}{\bf.}}}{\end{preremark}}

\newtheorem{preexample}{{\bf Example}}

\newtheorem{preproblem}{{\bf problem}}

\newtheorem{preproof}{{\bf Proof.}}

\newenvironment{proof}[1]{\begin{preproof}{\rm
               #1}\hfill{$\Box$}}{\end{preproof}}

\renewcommand{\thefootnote}

\begin{document}
\title{Remarks on Vector Space Generated by the Multiplicative Commutators of a Division Ring}
\author{M. Aaghabali, Z. Tajfirouz\\
{\footnotesize{\em School of Mathematics, University of Edinburgh,
Edinburgh EH9 3JZ, Scotland}}}
\footnotetext{E-mail Addresses: {\tt mehdi.aaghabali@ed.ac.uk}, {\tt z\_tajfirouz@yahoo.com}}
\footnotetext{The first author acknowledges the ERC grant 320974.}
\date{}
\maketitle
\begin{quote}
{\small \hfill{\rule{13.3cm}{.1mm}\hskip2cm} \textbf{Abstract}
Let $D$ be a division ring with centre $F.$ An element of the form $xyx^{-1}y^{-1}\in D$ is called a multiplicative commutator. Let $T(D)$ be the vector space over $F$ generated by all multiplicative commutators in $D.$ In~\cite{aa1}, authors have conjectured that every division ring is generated as a vector space over its centre by all of its multiplicative commutators. In this note it is shown that if $D$ is centrally finite, then the conjecture holds.
\vspace{1mm} {\renewcommand{\baselinestretch}{1}
\parskip = 0 mm

\noindent{\small {\it AMS Classification}: 17A01, 17A35.}}

\noindent{\small {\it Keywords}: Division Ring, Commutators.}}

\vspace{-3mm}\hfill{\rule{13.3cm}{.1mm}\hskip2cm}
\end{quote}
\section{Introduction}
  Throughout this paper $D$ is a division ring with centre $F.$ An element of the form $xyx^{-1}y^{-1}\in D$ is called a {\it multiplicative commutator}, and $[D,D]$ denote the {\it additive commutator} subgroup of $D$. Also we denote
by $T(D)$ the vector space generated by the set of all multiplicative commutators of $D$ over $F.$ An element $a\in D$ is said to be {\it algebraic} over $F$ if $a$ satisfies a non-zero polynomial in $F[x].$ A set $S\subseteq D$ is called algebraic if each of its elements is algebraic over $F.$ When $K$ is a finite-dimensional extension of $F,$ then we denote by $Tr_{K/F},$ the regular trace of $K$ over $F.$ If $a\in D,$ then $F(a)$ denotes the subfield of $D$ generated by $F$ and $\{a\}.$

The division ring generated by additive commutators or multiplicative commutators of $D$ is the whole $D$~\cite[pp. 205, 211]{lam}. In the algebraic and zero characteristic case, it was proved that $D$ is generated as a vector space over the centre   by the union of its additive commutators and the unity, see~\cite{akb}. In~\cite{aa1}, the first author and his colleagues study the $F$-vector space $T(D)$ generated by the set of multiplicative commutators $\{xyx^{-1}y^{-1}| x,y\in D^*\}.$ They prove that $T(D)$ contains all separable elements of $D$ and if $T(D)$ is radical over $F,$ then $D=F.$ Furthermore, if ${\rm dim}_F T(D) <\infty,$ then ${\rm dim}_F D <\infty.$ They also prove that if $D$ is an algebraic division ring over its centre with ${\rm char}(D) = 0,$ then $T(D)=D.$ The following theorem is crucial in their studies.
\begin{thm} {\rm \cite[p. 156]{rowen}}\label{rowen}
Let $D$ be a division ring and $K$ be a subfield of $D.$ For $a\in D,$ if ${\rm dim}_KK[a]\geq n$, then for any distinct elements $\alpha_1,\dots, \alpha_n\in Z(D)$, $(a-\alpha_1)^{-1},\dots, (a-\alpha_n)^{-1}$ are linearly independent.
\end{thm}

They then conjecture that a division ring is generated by all multiplicative commutators as a vector space over its centre, i.e., $T(D) = D$ for any arbitrary division ring. Recently, Hazrat has shown that if $D=L((x,\sigma))$ is the formal Laurent series in which $L/F$ is a field extension and $\sigma\in Gal(L/F)$ is of infinite order, then ${\rm dim}_{T(D)}D=\infty$~\cite{hazrat}. However, this example disproves the conjecture in general case, but still one can assume that conjecture holds for algebraic division rings. In fact, the case which still has not received answer is about algebraic division rings with characteristic $p>0.$ In this note, we show that the conjecture is true for centrally finite division rings, i.e. every division ring finite dimensional over its centre could be recovered as a vector space by its all multiplicative commutators.


We start with following result.

\begin{thm}\label{thm}
Let $D$ be an algebraic non-commutative division ring over its centre $F.$ Then $T(D)$ is a non-central Lie ideal in $D.$
\end{thm}
\begin{proof}{First note that if $F$ is finite, then $D$ is commutative, a contradiction. Hence we may assume that $F$ is infinite and let $a\in D,~x\in T(D)$ be arbitrary. We show that $ax-xa\in T(D).$ If $ax=xa,$ then we have nothing to prove. Otherwise, for every $r\in F$ we have
$$
\begin{array}{lll}0\neq r(axa^{-1}-x)=raxa^{-1}+ax-ax-rx\\~~~~~~~~~~~~~~~~~~~~~~~=axa^{-1}(a+r)-(a+r)x\\~~~~~~~~~~~~~~~~~~~~~~~=(axa^{-1}-(a+r)x(a+r)^{-1})(a+r).\end{array}
$$
Since  $T(D)$ is invariant under conjugation, for every $r\in F$ we obtain that $$(axa^{-1}-x)(a+r)^{-1}=r^{-1}(axa^{-1}-(a+r)x(a+r)^{-1})\in T(D).$$
Now, since $a$ is algebraic over $F$ we have ${\rm dim}_FF(a)=n<\infty.$ Thus by Theorem~\ref{rowen}, for every $r_1,\dots,r_n,~~(a-r_1)^{-1},\dots,(a-r_n)^{-1}$ are linearly independent and consequently are basis for $F(a)$ over $F.$ Therefore, for every $a\notin C_D(x),~~(axa^{-1}-x)F(a)\subseteq T(D).$ In particular, $ax-xa=(axa^{-1}-x)a\in T(D).$ This completes the proof.
}
\end{proof}

\begin{rem}\label{rem1}
By a result due to Herstein~\cite[p. 9, Theorem 1.5 ]{her}, every non-central Lie ideal of a division ring which is not $4$-dimensional over a field of characteristic $2,$ contains the additive commutator subgroup of the division ring. Combining this fact with previous theorem we find that if $D$ is a division ring algebraic over its centre $F,$ then $[D,D]\subseteq T(D)$ provided that either ${\rm char}(D)\neq 2$ or ${\rm dim}_FD>4.$
\end{rem}

\begin{rem}\label{rem}
Let $D$ be a non-commutative division ring finite dimensional over its center $F,$ and let $K=F(a)$ be the separable maximal subfield of $D$~\cite[P. 244]{lam}. It is known that every element of a division ring that commutes with all multiplicative commutators, is central~\cite[p. 210]{lam}. By this fact,  $T(D)$ is not commutative. Then since $T(D)$ contains all separable elements of $D,$ one can obtain that $K\varsubsetneq T(D)$~\cite[Theorem 2.3]{aa1}. Hence, ${\rm dim}_FT(D)>{\rm dim}_FK.$
\end{rem}

It is known that if $D$ is a centrally finite division ring, then $[D,D]$ is a hyperspace. Indeed, let $D$ be a division ring of dimension $n^2$ over its centre $F$ and suppose that $K$ is a splitting field for $D,$ i.e. $D\otimes_FK\simeq M_n(K).$ Now, looking at $[D,D]$ as a right $F$-module we know that
$$\begin{array}{lll}{\rm dim}_F[D,D]={\rm dim}_K([D,D]\otimes_FK)\\~~~~~~~~~~~~~~~={\rm dim}_K[D\otimes_FK,D\otimes_FK]\\~~~~~~~~~~~~~~~={\rm dim}_K[M_n(K),M_n(K)]={\rm dim}_KSl_n(K)=n^2-1.\end{array}$$

\begin{thm}
Let $D$ be a division ring finite dimensional over its centre $F.$  Then $D$ is generated as a vector space by all of its multiplicative commutators over $F,$ i.e. $D=T(D).$
\end{thm}
\begin{proof}{Assume that $F$ is finite. Then $D$ is finite and by the Wedderburn's Little Theorem must be commutative and there is nothing to prove. So, we may assume that $F$ is infinite.  Put $M=\{x\in D| Tr_{F(x)/F}(x)=0\}.$ Clearly, $M$ is a subspace of $D$ that contains $[D,D].$ Since $[D,D]$ is a hyperspace easily it is seen that either $[D,D]=M$ or $M=D.$ The latter case cannot occur, because by the Noether-Jacobson Theorem~\cite[p. 244]{lam}, $D$ contains some separable elements, say $b.$ Hence, there would be existing an element $c\in F(b)$ whose trace is non-zero~\cite[p. 86]{mor}. Therefore, $x\in [D,D]$ if and only if $Tr_{F(x)/F}(x)=0.$

 First assume that either ${\rm dim}_FD\neq4$ or ${\rm char}(F)\neq2.$ Combining Remark~\ref{rem1} with the fact that the additive group $[D,D]$ is a hyperspace, we obtain either $[D,D]=T(D)$ or $D=T(D).$  Note that $T(D)=[D,D]$ could not be the case, because as mentioned in Remark~\ref{rem} $T(D)$ admits some separable elements with non-zero trace while as the first paragraph of the proof $[D,D]$ is the set of all elements of trace zero.

 Now, assume that ${\rm dim}_FD=4$ and ${\rm char}(F)=2.$ By Remark~\ref{rem}, ${\rm dim}_FT(D)=3$ or ${\rm dim}_FT(D)=4.$ If ${\rm dim}_FT(D)=4$ the proof is completed. By the contrary suppose that ${\rm dim}_FT(D)=3.$ Observe that if for every $x\in D,~ x^{-1}\in T(D),$ then $D=T(D)$ and we are done. So, pick $x\in D$ such that $x^{-1}\notin T(D).$ Since $T(D)$ is a Lie ideal, for every $a\in T(D)$ we have $ax-xa\in T(D)\setminus F.$ This shows that $L:=[D,D]\cap T(D)\nsubseteq F.$ It is not hard to verify that $L$ is invariant under conjugation (in fact, if $x\in [D,D],$ then for every $a\in D^*$ we have $axa^{-1}-x=(ax)a^{-1}-a^{-1}(ax)\in [D,D]$). Let $b$ be a non-central separable element in $D.$ Then since $F(b)$ is a maximal subfield containing some elements with non-zero trace, we find that $F(b)\nsubseteq L.$ Therefore, argument similar to the proof of Theorem~\ref{thm} shows that there exists an element $x\in L$, which does not commute with $b,$ and $(bxb^{-1}-x)F(b)\subseteq L.$ Therefore, $L\cap Lb\neq 0,$ and in particular ${\rm dim}_F(L\cap La)\geq {\rm dim}_FF(b)=2.$ Also, $b\in Lb\setminus L$ implies that $L\cap Lb\varsubsetneq Lb,$ and ${\rm dim}_FLb={\rm dim}_FL=3.$ Thus $[D,D]\cap T(D)=L=[D,D].$ Combining this fact with contrary assumption ${\rm dim}_FT(D)=3$ yields that $T(D)=[D,D].$ Now, repeating aforementioned arguments achieves contradiction. This completes the proof.}
\end{proof}
\begin{rem}
Note that by~\cite[Theorem 2]{akb} if $D$ is a division ring algebraic over its centre $F$ with characteristic zero, then ${\rm dim}_FD/[D,D]\leq 1.$ Hence, applying similar argument as above yields that if $D$ is algebraic over its centre with characteristic zero, then one can recover $D$ from $T(D).$
\end{rem}


\end{document}